\documentclass[
12pt, 
a4paper, 
oneside, 
headinclude,footinclude, 
]{scrartcl}

\usepackage[
nochapters, 
pdfspacing, 
dottedtoc 
]{classicthesis} 
\usepackage{amsopn}
\usepackage[T1]{fontenc} 
\usepackage{multirow}
\usepackage[utf8]{inputenc} 
\usepackage{hhline}
\usepackage{graphicx} 
\graphicspath{{Figures/}} 
\usepackage{indentfirst}
\usepackage{enumitem} 
\usepackage{savesym}
\usepackage{mathrsfs}
\usepackage{geometry}
\usepackage{hyperref}
\hypersetup{colorlinks=true, citecolor=blue}
\usepackage{enumitem}
\usepackage{mathtools}
\usepackage[dvipsnames]{xcolor}
\usepackage{float}
\usepackage{longtable}
\usepackage{dirtytalk}

\usepackage[
    backend=bibtex,
    style=numeric,
    natbib=false,
    url=false, 
    doi=false,
    eprint=false,
    maxnames=50
]{biblatex}

\usepackage{xcolor}
\usepackage[skins]{tcolorbox}
\usepackage{cjhebrew}
\usepackage{esint}
\usepackage{subfig} 

\usepackage{amsmath,amssymb,amsthm,amsfonts} 

\usepackage{varioref} 

\usepackage{thmtools}
\usepackage{thm-restate}

\newcommand{\vertiii}[1]{{\left\vert\kern-0.25ex\left\vert\kern-0.25ex\left\vert #1 
    \right\vert\kern-0.25ex\right\vert\kern-0.25ex\right\vert}}

\theoremstyle{plain}
\newtheorem{teorema}{Theorem}[section]
\newtheorem{proposizione}[teorema]{Proposition}
\newtheorem{lemma}[teorema]{Lemma}
\newtheorem{corollario}[teorema]{Corollary}
\newtheorem*{theorem*}{Theorem}

\theoremstyle{definition}
\newtheorem{definizione}[teorema]{Definition}

\theoremstyle{Question}

\theoremstyle{plain}

\theoremstyle{plain}

\theoremstyle{plain}

\theoremstyle{definition}

\theoremstyle{remark}

\theoremstyle{remark}
\newtheorem{osservazione}{Remark}[section]

\theoremstyle{remark}

\newcommand{\N}{\mathbb{N}}

\newcommand{\R}{\mathbb{R}}

\newcommand{\HH}{\mathbb{H}}

\newcommand{\Haus}{\mathscr{H}}

\newcommand{\res}

\newcounter{const}

\newcounter{eps}

\DeclareMathOperator*{\diam}{diam}

\hypersetup{
colorlinks=true, breaklinks=true,bookmarksnumbered,
urlcolor=webbrown, linkcolor=RoyalBlue, citecolor=webgreen, 
pdftitle={}, 
pdfauthor={\textcopyright}, 
pdfsubject={}, 
pdfkeywords={}, 
pdfcreator={pdfLaTeX}, 
pdfproducer={LaTeX with hyperref and ClassicThesis} 
} 

\title{\large \normalfont\spacedallcaps{On sets with unit Hausdorff density in homogeneous groups}} 

\author{\spacedlowsmallcaps{Antoine Julia\textsuperscript{*}, Andrea Merlo\textsuperscript{**}}} 

\date{\today} 

\begin{document}


\renewcommand{\sectionmark}[1]{\markright{\spacedlowsmallcaps{#1}}} 
\lehead{\mbox{\llap{\small\thepage\kern1em\color{halfgray} \vline}\color{halfgray}\hspace{0.5em}\rightmark\hfil}} 

\pagestyle{scrheadings} 


\maketitle 

\setcounter{tocdepth}{2} 
\tableofcontents



\paragraph*{Abstract} 
It is a longstanding conjecture that given a subset $E$ of a metric space, if $E$ has finite Hausdorff measure in dimension $\alpha$
 and $\mathscr{H}^\alpha\llcorner E$ has unit density almost everywhere, then $E$ is an $\alpha$-rectifiable set. We prove this conjecture under the assumption that the ambient metric space is a homogeneous group with a \textit{smooth-box} norm.


{\let\thefootnote\relax\footnotetext{* \textit{D\'epartement de Math\'ematiques d'Orsay, Universit\'e Paris-Saclay, Orsay, France}\\
**\textit{D\'epartement de Math\'ematiques, Universit\'e de Fribourg, Fribourg, Suisse}}}

\section{Introduction}

A subset $E$ of a metric space $(X,d)$ is \textbf{$m$-rectifiable} if it can be covered up to an $\mathscr{H}^m$-null set, here as usual $\mathscr{H}^\alpha$ denotes the Hausdorff measure of dimension $\alpha\ge 0$, see \cite[\S 2.10.2]{Federer1996GeometricTheory}, by a countable union of Lipschitz images of subsets of $\R^m$. It was proved by B.~Kirchheim in \cite{kircharea} that such a set $E$ has \textbf{Hausdorff density} equal to one at almost every point, i.e.:
\[
	\Theta^m(E,x) := \lim_{r\to 0} \dfrac{\mathscr{H}^m(E\cap B(x,r))}{(2r)^m} = 1,
\]
for $\mathscr{H}^m$-almost all $x\in E$, where $B(x,r)$ denotes the closed unit ball in $(X,d)$ with center $x$ and radius $r$. It is natural to ask the converse question: does the unit Hausdorff density property imply rectifiability of $E$? A positive answer assuming that $(X,d)$ is the euclidean space, was given by P.~Mattila in the fundamental work \cite{Mattila1975hausdorff}, following partial results by A.~S.~Besicovitch \cite{Besicovitch1927fundamental} and J.~M.~Marstrand \cite{Marstrand1961D2inR3,Marstrand1964Regularsubsets}. More precisely, the following result holds~:
\begin{teorema}[P.~Mattila, 1975]\label{thm-euclidean}
	Fix $\alpha\ge 0$, if $E\subset \R^n$, endowed with the euclidean distance, is such that $\mathscr{H}^\alpha(E) <+\infty$ and $\Theta^\alpha(E,x) = 1$ for $\mathscr{H}^\alpha$-almost all $x\in E$, then $\alpha$ is an integer and $E$ is $\alpha$-rectifiable.
\end{teorema}
This result is surprising, in that one is able to infer strong geometric property (the existence of flat, $\alpha$-dimensional tangents to $E$ at almost all points) from an \textit{a priori} much less geometric condition. 
 Note that the unit density condition can be relaxed, to bounds on the lower and upper densities of $E$, as was shown by D.~Preiss in \cite[Corollary 5.4(3)]{Preiss1987GeometryDensities}.

\medskip

In general metric spaces, and for sets of dimension $1$ the question was answered by D.~Preiss and J.~Ti\v ser \cite{PreissTiser-Besicovitch}. Their proof uses a variant of the argument by Besicovitch, which is specific of dimension $1$ and the fact that a continuum with finite $\mathscr{H}^1$ measure is rectifiable. The extension to higher dimensions is still widely open, even in $\R^n$ with an non-euclidean norm. In this paper, we treat the case in which the ambient metric space is a homogeneous group (for instance, any Carnot group, see Section \ref{sec:preliminaries} for definitions) and we prove

\begin{teorema}\label{thm:main}
  Given a homogeneous group $\mathbb{G}$ there is a homogeneous norm $||\cdot||$ on $\mathbb{G}$, 
  for which, whenever $\alpha\ge 0$ and $E\subset \mathbb{G}$ is such that $\mathscr{H}^\alpha(E) <+\infty$ the following are equivalent
  \begin{itemize}
      \item[(i)] $\alpha\in\N$ and $E$ is covered up to an $\mathscr{H}^\alpha$-null set by countably many Lipschitz images of compacts of $\R^\alpha$;
      \item[(\hypertarget{two}{ii})] $\Theta^\alpha(E,x) = 1$ for $\mathscr{H}^\alpha$-almost all $x\in E$.
  \end{itemize}
\end{teorema}

As in the Euclidean case, the hypothesis of the theorem implies that $\alpha$ is an integer, but here it also implies that there exists an $\alpha$ dimensional subgroup in the first layer of $\mathbb{G}$. For instance in the first Heisenberg group $\mathbb{H}^1$, the only possible values for $\alpha$ are $0$ and $1$, while the Hausdorff dimension of $\mathbb{H}^1$ is $4$.

\medskip

Our proof mainly relies on the following observations. First, the Carath\'eodory construction which yields the measure $\Haus^\alpha$ is obtained with coverings with \emph{ arbitrary closed sets}. This roughly means that if the $\Haus^\alpha$-measure of a set is finite, fixed a set $S$ of very small diameter, say $2r$, containing an $x\in E$, one must have that $\mathscr{H}^\alpha(E\cap S)\lesssim (2r)^\alpha$ , see Theorem \ref{prop:federer-density}. On the other hand, if one requires that item
(\hyperlink{two}{ii}) of Theorem \ref{thm:main} holds, one infers also that, $\mathscr{H}^\alpha(B(x,r)\cap E)\gtrsim (2r)^\alpha$, which means that if there exists a set containing the unit ball and with diameter still smaller than $2r$, let us denote is by $G$, we must have that $\Haus^\alpha(x*\delta_r(G)\setminus B(x,r)))/(2r)^\alpha$ is very small.

The second observation is that if one chooses properly the metric, then one can force the set $E$ with unitary Hausdorff density to be concentrated along specific directions. 
Our proof of this result relies on the fact that the unit ball for the norm $||\cdot||$ is very far from being isodiametric, more precisely we construct a set $G(0,1)$ with diameter $2$ and containing the unitary ball such that for any $x\in B(0,1)$ for which $\lvert x_1\rvert<1$, we have that $B(0,1)\cap x_1+V_2\oplus \ldots \oplus V_\kappa$ is \emph{compactly} contained in the interior of $G(0,1)\cap x_1+V_2\oplus \ldots \oplus V_\kappa$. This condition guarantees the existence of a bi-Lipschitz function $f$ from the first layer $V_1$ of the algebra of $\mathbb{G}$ endowed with the Euclidean metric to $\mathbb{G}$ such that $E$ is contained  in the image $f$. Finally, noting that one can choose the bi-Lipschitz constant of $f$ to be as close to $1$ as one wants, and this allows us in the end to apply Theorem \ref{thm:main}. For the details, see the proof of Theorem \ref{thm:main} at page \pageref{proofmain}.

\medskip

 A similar construction can be achieved in many other cases, for instance for the Koranyi norm in Heisenberg groups, see Proposition \ref{GforKoranyi}, but we choose to give proofs for the smooth-box distance, see  \eqref{eq:smoothbox}, since on the one hand it is a very natural metric widely used in the literature and on the other the construction of the set $G(0,1)$ can be provided with the same argument in any group. On the other hand, we stress that our construction does not work for the Carnot-Carathéodory distance in the Heisenberg group, even though the unit ball for this distance is also far from being isodiametric. 

\medskip

The isodiametric problem in a homogeneous group consists in maximizing the volume over all compact sets with diameter at most $2$. As discussed in the previous paragraph, it is clear that this problem is deeply connected to that of recognising what sets in said homogeneous groups can have unit density for the Hausdorff measure, see Theorem \ref{thm:measuresrectifiable}.
In a finite dimensional vector space, it is known that the unit ball is the unique maximizer, up to translations \cite{Bieberbach1915Isodiametric, Melnikov1963Isodiametric}. If $\mathbb{G}$ is a homogeneous group which is not a vector space we expect that the unit ball is not isodiametric for any homogeneous norm. Indeed, there are many examples of balls in homogeneous groups which are not isodiametric \cite{Metelichenko-Thesis,Rigot2011Isodiametric, LeonardiRigotVittone}, but we do not know of a proof which would work in general, not even if the group is the parabolic plane. We discuss this and give examples in Section \ref{sec:isodiametric}.

\medskip

  It is worth noting that our result is also relevant when passing from homogeneous groups to the more general context of metric space.  In \cite{LeDonne2011Unique} it is shown that the Gromov-Hausdorff tangent of a geodesic metric space with unique Gromov-Hausdorff tangents is a Carnot group, which in particular is a homogeneous group. Furthermore, this result has an extension to non necessarily geodesic metric spaces: it is proved in  \cite[Theorem 1.6]{LeDonneNiGo2021dilations} that if a doubling metric measure space has unique tangents at almost every point, then almost all these tangents are metric Lie groups admitting dilations. In Section~\ref{sec:groupwithdilations}, we will show that these groups are biLipschitz equivalent to groups for which the conclusion of Theorem~\ref{thm:main} holds.


\medskip

The structure of this paper is the following: in Section \ref{sec:preliminaries}, we present the notations and the basic results that we will use.  Section \ref{sec:surfacemeasure} contains a discussion on the comparison of the notions of surface measures on intrinsic rectifiable sets in Carnot groups. This motivates Section \ref{sec:isodiametric}, in which we review the isodiametric problem in homogeneous groups. Section \ref{sec:proof} is dedicated to the proof of Theorem \ref{thm:main}. In section~\ref{sec:groupwithdilations}, we adress the more general notion of groups with dilations. Finally in Section \ref{sec:h1} we focus on the first Heisenberg group and discuss the density problem in that setting.

\section{Preliminaries}\label{sec:preliminaries}
A homogeneous group $\mathbb{G}$ is a nilpotent Lie group which is graded, see \cite{LD17}, for precise definitions, in particular it can be decomposed as a direct sum:
\[
     \mathbb{G} = V_1 \oplus \dots \oplus V_\kappa,
   \]
where the finite dimensional vector spaces $V_\ell$ are called \textbf{ layers} and $\kappa$ is the \textbf{step} of $\mathbb{G}$. It will always be assumed that $V_1\neq \{0\}$, see Remark \ref{rem:V1neq0}.

\[
    \delta_\lambda (x_1,x_2,\dots,x_\kappa) = (\lambda x_1, \lambda^2 x_2,\dots , \lambda^\kappa x_\kappa).
\]
By the Baker-Campbell-Hausdorff formula, the group operation can be written as follows, for $x,y\in \mathbb{G}$ and $\ell \in \{1,\dots, \kappa\}$:
\[
     (x*y)_\ell = x_\ell + y_\ell + Q_\ell(x,y),
\]
where $Q_\ell$ is a $V_\ell$ valued polynomial function in the first $\ell-1$ components of $x$ and $y$, in particular, $Q_1=0$. There also holds: $Q_\ell(x,0)=Q_\ell(x,x)=0$ and $Q_\ell(\delta_\lambda x, \delta_\lambda y ) =\lambda^\ell Q_\ell(x,y)$ for $\lambda >0$.

\begin{definizione}[Homogeneous norm, homogeneous distance]
A distance $d:\mathbb{G}\times \mathbb{G}\to \R$ is \textbf{homogeneous} and \textbf{left-invariant} if for any $x,y\in \mathbb{G}$ there holds
\begin{itemize}
    \item[(i)] $d(\delta_\lambda x,\delta_\lambda y)=\lambda d(x,y)$ for any $\lambda>0$,
    \item[(ii)] $d(z *x,z* y)=d(x,y)$ for any $z\in \mathbb{G}$.
  \end{itemize}
  It is equivalent to define a homogeneous distance $d$ and a \textbf{homogeneous norm} $||\cdot||$, as can be seen by writing
  \[
       || x|| := d(0,x), \text{ or conversely } d(x,y):= ||x^{-1}*y||.
  \]
\end{definizione}

In the following, $d$ is a left-invariant homogeneous metric on $\mathbb G$ and the corresponding closed ball with center $x\in \mathbb G$ and radius $r>0$ is denoted by $\mathrm{B}_d(x,r)$. We will drop the subscript $d$ when there is no ambiguity. As the first layer $V_1$ is assumed to be non degenerate, the balls $B_d(x,r)$ have diameter $2r$. The following result, the elementary proof of which we learned in Metelichenko's thesis \cite[Lemma~3.17]{Metelichenko-Thesis}, will be very useful.
\begin{proposizione}
  If $||\cdot||$ is a homogeneous norm on $\mathbb{G}$ then its restriction $|\cdot|$ to $V_1$ is a vector space norm. Also, if $\pi_1$ denote the projection $\mathbb{G}\to V_1$, $(x_1,\dots,x_\kappa) \mapsto x_1$. Then $\pi_1$ is exactly $1$-Lipschitz with the norm $|\cdot|$ on $V_1$.
\end{proposizione}
\begin{proof}
  As $\pi_1(x_1,0,\dots,0) = x_1,$ the projection $\pi_1$ cannot have Lipschitz constant smaller than one. In order to prove the opposite bound, note that $\pi_1$ is a group morphism, thus, it suffices to show that $\pi_1(x) \le ||x||$ for $x\in \mathbb{G}$. Fix $x=(x_1,\dots,x_\kappa)\in\mathbb{G}$. On the one hand, there holds $||\delta_{1/2}(x*x)|| = ||x*x||/2 \le ||x||$, on the other hand:
  \[
     \delta_{1/2}(x*x)= (x_1,2^{-1}x_2,\dots, 2^{1-\kappa} x_\kappa).
   \]
   Iterating this operation, and noting that $\delta_{1/2^n}((x*x)*\dots *(x*x)) \to (x_1, 0,\dots, 0)$, one gets $||(x_1,0,\dots,0)||\le ||x||$. Finally, given $x_1$ and $y_1$  in $V_1$ there holds
   \[
     |x_1+y_1|= ||\pi_1((x_1,0,\dots,0)*(y_1,0,\dots,0)) || \le ||(x_1,0,\dots,0)|| + ||(y_1,0,\dots,0)|| = |x_1|+|y_1|.
   \]
  and thus $|\cdot|$ is a vector-space norm on $V_1$ (homogeneity and positivity are immediate).
\end{proof}

\begin{definizione}[Hausdorff and spherical measures]\label{def:HausdorffMEasure}
For $\alpha\ge 0$, define the $\alpha$-dimensional \textbf{Hausdorff measure}  relative to the left invariant homogeneous distance $d$ 
$$
\mathscr{H}^\alpha_d(E)=\sup_{\delta>0}\mathscr{H}^\alpha_{d,\delta}(E):=\sup_{\delta>0}\inf \bigg\{\sum_{j=1}^{\infty} (\diam S_j)^\alpha:\ E \subseteq \bigcup_{j=1}^{\infty} S_j,\, S_j \text{ closed},\  \diam S_j\leq \delta\bigg\}.
$$

Along with the Hausdorff measure, it is useful to study other surface measure, we will focus on two of them: the $\alpha$-dimensional \textbf{ spherical measure}:
$$
\mathscr{S}^{\alpha}(A):=\sup_{\delta>0}\inf\bigg\{\sum_{j=1}^\infty  (2r_j)^\alpha:A\subseteq \bigcup_{j=1}^\infty \mathrm{B}_d(x_j,r_j),~r_j\leq\delta\bigg\}
$$
and the $\alpha$-dimensional \textbf{ centered spherical measure}:
$$
\mathscr{C}^{\alpha}(A):=\underset{E\subseteq A}{\sup}\,\,\mathscr{C}^\alpha_0(E),
$$
where
$$
\mathscr{C}^{\alpha}_{0}(E):=\sup_{\delta>0}\inf\bigg\{\sum_{j=1}^\infty  (2r_j)^\alpha:E\subseteq \bigcup_{j=1}^\infty B_d(x_j,r_j),~ x_j\in E,~r_j\leq\delta\bigg\}.
$$
\end{definizione}
These are all outer measure, and thus define Borel regular measures. The centered spherical measure is studied in see \cite[Proposition~4.1]{EdgarCentered}. Furthermore, they are comparable,  
as can be seen by covering any set of diameter $r$ by a ball of radius $2r$ centered at any of its points, which yields the following relations:
 \begin{equation}\label{eq:missuperficie}
     \mathscr{H}^h\leq \mathscr{S}^h\leq\mathscr{C}^h \leq 2^h\mathscr{H}^h.
 \end{equation}

\begin{definizione}[Lower and upper densities]\label{def:densities}
If $\phi$ is a Radon measure on $\mathbb{G}$, and $\alpha>0$, define the \textbf{lower} and \textbf{upper $\alpha$-densities of $\phi$ at $x$} by
\[
\Theta_*^{\alpha}(\phi,x):=\liminf_{r\to 0} \frac{\phi(B(x,r))}{(2r)^{\alpha}},\qquad \text{and}\qquad \Theta^{\alpha,*}(\phi,x):=\limsup_{r\to 0} \frac{\phi(B(x,r))}{(2r)^{\alpha}},
\]
 If the upper and lower densities coincide, their value is called the \textbf{density of $\phi$ at $x$}.
 If $\phi= \mathscr{H}^\alpha \llcorner E$, where $E$ is a subset of $\mathbb{G}$, its density (resp.~upper, lower densities) is also called \textbf{Hausdorff density of $E$} and denoted $\Theta^\alpha(E,x)$.
\end{definizione}

The following fundamental result on Hausdorff measures lies at the core of this paper. It is a consequence  of \cite[Theorems 2.10.17 and 2.10.18]{Federer1996GeometricTheory}, see also Theorem 2.1 and Lemma 3.1 in \cite{JNGV22coarea}.
\begin{proposizione}\label{prop:federer-density}
  Suppose that $E\subset \mathbb{G}$ satisfies $0<\mathscr{H}^\alpha(E) <+\infty$. Then for $\mathscr{H}^\alpha$-almost every $x\in E$, there holds:
  \begin{equation}\label{eq:federer-density}
      \lim_{\delta \to 0} \sup \bigg \{  \dfrac{\mathscr{H}^\alpha(E\cap S)}{(\diam S)^\alpha}, x\in S, \diam S <\delta\bigg\} \le 1.
    \end{equation}
    Furthermore, if $Q$ is the homogeneous dimension of $\mathbb{G}$, there holds
    \begin{equation*}
      \lim_{\delta \to 0} \sup \bigg \{  \dfrac{\mathscr{H}^Q(S)}{(\diam S)^Q}, x\in S, \diam S <\delta \bigg\} = 1.
    \end{equation*}
\end{proposizione}

There is a degenerate case, which needs to be ruled out:
\begin{proposizione}\label{rem:V1neq0}
If $V_1=\{0\}$, then sets with unit Hausdorff density are necessarily of dimension $0$.
\end{proposizione}
\begin{proof}
    First, note that a ball of radius $r$ in such a group necessarily has diameter strictly less than $2r$. Indeed letting $\ell$ be the index of the first non-trivial layer, $||\cdot||^{\ell}$ is a norm on the homogeneous group $\mathbb{G}'$, obtained from $\mathbb{G}$ by changing the homogeneity of the dilations. The diameter of the unit ball in $\mathbb{G}'$ is $2$ because the group has non trivial first layer, thus its diameter in $\mathbb{G}$ is $(2)^{1/\ell}$, which is smaller than $2$. By homogeneity, there exists $\varepsilon>0$ such that $\diam(B(x,r)) \le (1-\varepsilon) (2r)$ for all $x\in \mathbb{G}$ and all $r>0$. 
    Using the first statement of Proposition \ref{prop:federer-density} at a point where it holds, for small enough $r>0$, one gets
    \[
        \mathscr{H}^\alpha(E\cap B(x,r)) \le (1+\varepsilon)^\alpha \diam B(x,r)^\alpha < (1+\varepsilon)^\alpha(1-\varepsilon)^\alpha(2r)^\alpha = (1-\varepsilon^2)^\alpha (2r)^\alpha.
    \]
    Dividing by $(2r)^\alpha$ and letting $r$ go to zero contradicts the Hausdorff density hypothesis.
\end{proof}

\section{Comparison of surface measures in Carnot groups}\label{sec:surfacemeasure}

Throughout this section $\mathbb{G}$ will be a fixed Carnot group and $d$ a fixed homogeneous left-invariant distance on $\mathbb{G}$.

\begin{definizione}[$C^1_{\mathrm H}(\mathbb{G}, \mathbb{G}')$-rectifiable sets.]\label{def:C1h}
Let $\mathbb G'$ be another Carnot groups endowed with  the left-invariant homogeneous distance $d'$. Let $\Omega\subseteq \mathbb G$ be open and let $f:\Omega\to\mathbb G'$ be a function. We say that $f$ is \textbf{Pansu differentiable at $x\in\Omega$} if there exists a homogeneous homomorphism $Df_x:\mathbb G\to\mathbb G'$ such that
$$
\lim_{y\to x}\frac{d'(f(x)^{-1}\cdot f(y),Df_x(x^{-1}\cdot y))}{d(x,y)}=0.
$$
Moreover we say that $f$ is \textbf{of class $C^1_{\mathrm H}$ in $\Omega$} if the map $x\mapsto Df_x$ is continuous from $\Omega$ to the space of homogeneous homomorphisms from $\mathbb G$ to $\mathbb G'$. 
Given an arbitrary Carnot group $\mathbb G$, we say that $\Gamma\subseteq \mathbb G$ is a \textbf{ $C^1_{\mathrm H}$-submanifold} of $\mathbb G$ if there exists a Carnot group $\mathbb G'$ such that for every $p\in \Gamma$ there exists an open neighborhood $\Omega$ of $p$ and a function $f\in C^1_{\mathrm H}(\Omega;\mathbb G')$ such that 
\begin{equation}\label{eqn:RepresentationOfSigma}
\Gamma\cap \Omega =\{g\in\Omega:f(g)=0\},
\end{equation}
and $Df_p:\mathbb G\to\mathbb G'$ is surjective and ${\mathrm Ker}(Df_p)$ admits a complementary subgroup. In this case we say that $\Gamma$ is a \textbf{ $C^1_{\mathrm H}(\mathbb G,\mathbb G')$-submanifold}. 
Finally, we say that $\Gamma\subseteq \mathbb G$ is a \textbf{$(\mathbb G,\mathbb G')$-rectifiable set} if  $\mathscr{H}^{Q-Q'}(\Gamma) <+\infty$ and there exist countably many subsets $\Gamma_i$ of $\mathbb G$ that are $C^1_{\mathrm H}(\mathbb G,\mathbb G')$-submanifolds, such that 
$$
\mathscr{H}^{\mathcal{Q}-\mathcal{Q}'}\left(\Gamma\setminus\bigcup_{i=1}^{+\infty}\Gamma_i\right)=0.
$$
At $\mathscr{H}^{Q-Q'}\llcorner \Gamma$-almost every point $x$, $\Gamma$ admits a tangent $\mathbb{V}(x)$, which is a homogeneous subgroup of $\mathbb{G}$ with homogeneous dimension $Q-Q'$.
\end{definizione}

\begin{definizione}[Tangent measures] \label{def:TangentMeasure}
Let $\phi$ be a Radon measure on $\mathbb G$. For any $x\in\mathbb G$ and any $r>0$ we define the rescaled measure
$$
T_{x,r}\phi(E):=\phi(x*\delta_r(E)), \qquad\text{for any Borel set }E\subset \mathbb{G}.
$$
The \emph{$\alpha$-dimensional tangents} to $\phi$ at $x$ are the Radon measures $\nu$ which are weak limits of a sequence $r_i^{-\alpha} T_{x,r_i} \phi$ for $r_i\to 0$. The set of these measures is denoted by $\mathrm{Tan}_{\alpha}(\phi,x)$.
\end{definizione}

\begin{definizione}\label{prectifiable}
Let $\phi$ be a Radon measure in $\mathbb{G}$ and let $h\in\N$. We say that $\phi$ is $\mathscr{P}_h$-rectifiable if
\begin{itemize}
    \item[(i)] $0<\Theta^h_*(\mu,x)\leq\Theta^{h,*}(\mu,x)<\infty $ for $\phi$-almost every $x\in\mathbb{G}$;
    \item[(\hypertarget{ii}{ii})] for $\phi$-almost every $x\in\mathbb{G}$ there exists an homogeneous subgroup $\mathbb V(x)$ of $\mathbb{G}$ of Hausdorff dimension $h\in\N$ such that
    $\mathrm{Tan}_h(\phi,x)\subseteq \{\lambda \mathscr{H}^h\llcorner \mathbb V(x):\lambda>0\}$.
\end{itemize}
\end{definizione}

The various surface measures on $(\mathbb{G},\mathbb{G}')$-rectifiable sets are now well understood, as summed up in the following statement which is a combination of  \cite[Theorem 1.1]{AntonelliMerlo2021}, \cite[Proposition 5.2]{antonelli2020rectifiableA}, as well as the results of \cite{JNGV22coarea}.

\begin{teorema}\label{thm:measuresrectifiable}
Suppose that $\Gamma\subseteq\mathbb G$ is a Borel, $(\mathbb{G},\mathbb{G}')$-rectifiable set. Then $\mathscr{H}^{Q-Q'}\llcorner \Gamma$, $\mathscr S^{Q-Q'}\llcorner \Gamma$ and $\mathscr C^{Q-Q'} \llcorner \Gamma$ are $\mathscr{P}_{Q-Q'}$-rectifiable. Furthermore,  at almost every point $x\in \Gamma$
\begin{equation}\label{eq:centeredsphericaldensity}
 \Theta^{Q-Q'}(\mathscr{C}^{Q-Q'}\llcorner\Gamma,x)=1.
\end{equation}
Finally, letting $\mathbb{V}(x)$ be the $Q-Q'$ subgroup for which (\hyperlink{ii}{ii}) is satisfied
\begin{eqnarray}
\mathscr{H}^{Q-Q'}\llcorner \Gamma =
\mathcal{A}(\mathscr{C}^{Q-Q'}\llcorner \mathbb{V}(\cdot))\mathscr{C}^{Q-Q'}\llcorner \Gamma,\label{eq:HvsC}\quad\text{and}\quad
\mathscr{H}^{Q-Q'}\llcorner \Gamma = \mathcal{A}(\mathscr{S}^{Q-Q'}\llcorner\mathbb{V}(\cdot))\mathscr{S}^{Q-Q'}\llcorner \Gamma,\nonumber
\end{eqnarray}
where for a homogeneous subgroup $\mathbb{V}$ of $\mathbb{G}$ with dimension $Q-Q'$ and a Haar measure $\mu$ on $\mathbb{V}$ we write:
\begin{equation*}
\mathcal{A}(\mu) := \sup \{ (\diam S)^{-(Q-Q')}\mu (S), \ 0< \diam S < +\infty \}.
\end{equation*}
\end{teorema}

The above result tells us that the coefficients that allows to express the Hausdorff measure on smooth surfaces in terms of the centred Hausdorff measure or the spherical Hausdorff measure are almost everywhere the volume (with respect to the centred or the sperical measures) of the isodiametric set in the tangents.

\begin{corollario}\label{cor:isodiam}
Under the hypotheses of Theorem \ref{thm:measuresrectifiable} and at a point $x$ where its conclusions holds, suppose that  
\begin{equation}\label{eq:rectifHausdorffdensity}
\Theta^{Q-Q'}(\mathscr{H}^{Q-Q'}\llcorner \Gamma, x)=1.
\end{equation}
Then the intersection of $B(0,1)$ with $\mathbb{V}(x)$ is an isodiametric shape in the homogeneous group $\mathbb{V}(x)$ with the  metric given by the restriction of $d$.
\end{corollario}
\begin{proof}
Combining \eqref{eq:centeredsphericaldensity}, \eqref{eq:HvsC} and the Hausdorff density hypothesis \eqref{eq:rectifHausdorffdensity} yields 
\begin{equation}\label{eq:rectifdensity1}
1= \mathcal{A}(\mathscr{C}^{Q-Q'}\llcorner \mathbb{V}(x)) =\sup \Big \{ \dfrac{(\mathscr{C}^{Q-Q'}\llcorner \mathbb{V}(x) )(S)}{(\diam S)^{(Q-Q')}}, \ 0< \diam S < +\infty \Big \},
\end{equation}
but by the fact that $\mathscr{C}^{Q-Q'}\llcorner \Gamma$ is $\mathscr{P}^{Q-Q'}$-rectifiable and \eqref{eq:centeredsphericaldensity}, there also holds
\[
\dfrac{(\mathscr{C}^{Q-Q'}\llcorner \mathbb{V}(x) )(B(0,1))}{2^{(Q-Q')}} = \lim_{r\to 0}\dfrac{(\mathscr{C}^{Q-Q'}\llcorner \Gamma )(B(x,r))}{(2r)^{(Q-Q')}} = 1.
\]
So that the ball $B(0,1)$ realizes the supremum of the right-hand-side of \eqref{eq:rectifdensity1}. Therefore the restriction of the unit ball $B(0,1)\cap \mathbb{V}(x)$ is isodiametric in the group $\mathbb{V}(x)$.
\end{proof}

\section{The isodiametric problem in homogeneous groups}\label{sec:isodiametric}

Consider a homogeneous Lie group $\mathbb{G}$, endowed with the ''volume measure'' $\mathscr{H}^{Q}$ and a homogeneous norm $||\cdot||$. A set $E\subset \mathbb{G}$ is \textbf{isodiametric} if it maximizes the volume among sets with diameter at most two. Being isodiametric depends on the distance, but clearly not on the choice of the Haar measure. The goal of this section is not to find isodiametric sets in general, but rather to ask whether the unit ball is isodiametric or not. The reason for this is the following consequence of \ref{prop:federer-density} (or Lemma 3.1 in \cite{JNGV22coarea}) and is proved in the same way as Corollary~\ref{cor:isodiam}:
\begin{proposizione}
  If $Q$ is the homogeneous dimension of $\mathbb{G}$, then the following are equivalent:
  \begin{enumerate}[label=(\roman*)]
    \item $\Theta^Q(\mathscr{H}^Q,x) = 1$ for all $x\in \mathbb{G}$,
    \item The unit ball in $\mathbb{G}$ is isodiametric.
  \end{enumerate}
\end{proposizione}


If $\mathbb{G}$ is a finite dimensional vector space (i.e. a homogeneous group of step $1$), the unit ball is isodiametric, independently of the norm. There is a simple proof of this fact using the Brunn-Minkowski inequality: If a set $E\subset \mathbb{R}^n$ has diameter two for some norm on $\mathbb{R}^n$, there holds $E+(-E)\subset B(0,2)$, and as $\mathscr{H}^n(E)=\mathscr{H}^n(-E)$, by the Brunn-Minkowski inequality
\[
     2^n\mathscr{H}^n(E) \le \mathscr{H}^n(E+(-E)) \le \mathscr{H}^n(B(0,2)) = 2^n \mathscr{H}^n(B(0,1))
\]
And equality holds if and only if $E$ is a translate of the unit ball. There is also a Brunn-Minkowski inequality in homogeneous groups (see for instance \cite{Pozuelo2019BrunnMinkowski}), but a similar proof fails when the homogeneous dimension of the group is larger than its topological dimension, as the Brunn-Minkowski inequality in groups involves the topological dimension. 
Furthermore, in homogeneous groups of step $2$ or above, there always exists a distance for which the unit ball is not isodiametric (see \cite{Rigot2011Isodiametric}), for instance the smooth-box distance which will be defined in Section~\ref{sec:proof}. It seems that it should be the same for any distance, but we are not able to prove this at the moment. 

\medskip

Let us now focus on the simplest case for this problem: the parabolic plane. It is the homogeneous group $\R^2$, with the usual law $+$ (which is commutative) and dilations $\delta_\lambda:(x,t)\mapsto (\lambda x,\lambda^2).$ 
Fix a homogeneous distance $d$ on $\mathbb{W}$ and let $B_d$ be the associated ball. Without loss of generality, one can suppose that
\[ 
       B_d\cap \{(x,0), x\in \R\} = [\,-1,1\,] \times \{0\} \text{ and }  B_d\cap \{(0,t), t\in \R\} =  \{0\}\times [\,-1,1\,],
\]
therefore, for $x\in \R$ there holds $d((0,0),(x,0)) = |x|$ and for $t\in \R$, $d((0,0),(0,t)) = |t|^{1/2}$. Here are five examples of such distances:
\begin{enumerate}
  \item $||(x,t)||_\infty:= \max\{|x|,|t|^{1/2}\},$ (Smooth-box distance) 
  \item $||(x,t)||_4 := (|x|^4  + |t|^2)^{1/4}\},$ (''Koranyi'' distance)  
  \item $||\cdot||_{CC}$, (the restriction of the Carnot-Carathéodory distance in $\mathbb{H}^1$ after rescaling, see Chapter 3 of \cite{Metelichenko-Thesis}),
  \item $||(x,t)||_1 := |x|+ |t|^{1/2}$,
  \item $||(x,t)||_?:= \max\{ |x|, |t|^{1/2} - \text{sign}(t) x\}$.
\end{enumerate}
\begin{proposizione}
  If $||\cdot||$ is one of the above functions, then it is a homogeneous norm on the parabolic plane, and if $B$ is the associated unit ball, then it is not an isodiametric set for $||\cdot||$. 
\end{proposizione}
\begin{proof}
  The proof that $||\cdot||$ is a homogeneous norm is left to the reader. To check that the unit ball is not isodiametric, the usual technique is to find a set with larger volume, but the same diameter. This works only for $||\cdot||_\infty$, $||\cdot||_4$ and $||\cdot||_{CC}$. Indeed, for the last two norms, given any point on the boundary of the unit ball, there is another point at distance $2$ from it, so the ball cannot be enlarged as such. Still, it is not hard to find a better candidate for the isodiametric problem by hand. However there is a procedure which works for all five norms, which we learned in Metelichenko's PhD thesis \cite[Theorem~3.34]{Metelichenko-Thesis}:

  \medskip

  It is easy to check that $B$ contains a symmetric euclidean convex set $C$ with $\mathscr{H}^Q(C)>\mathscr{H}^Q(B)/2$. Denote by $\frac{1}{2} C$ the {\textit euclidean} contraction of $C$ by a factor $1/2$. Notice that $\diam \frac{1}{2} C \leq 1$, as $ \frac{1}{2}C + (-\frac{1}{2}C) \subset C$, by euclidean convexity and symmetry, so that $\diam \delta_2 (\frac{1}{2} C) \leq 2$. On the other hand
\[
  \mathscr{H}^Q(\delta_2 (\textstyle{\frac{1}{2}} C) ) = 2 \mathscr{H}^Q (C) > \mathscr{H}^Q (B).
\]
Thus $\delta_2 (\frac{1}{2} C)$ has a better isodiametric ratio than $B$ and $B$ is not isodiametric.
\end{proof}

\section{Sets with Hausdorff density $1$ in homogeneous groups}\label{sec:proof}

%
One can endow a homogeneous group $\mathbb{G}$ as above with a smooth-box homogeneous norm of the form
\begin{equation}\label{eq:smoothbox}
    ||(x_1,\dots,x_\kappa)|| := \max \{ \epsilon_\ell |x_\ell|^{1/\ell},\ \ell= 1,\dots,\kappa\},
\end{equation}
where $|x_\ell|$ stands for the euclidean norm of the vector $x_\ell\in V_\ell$ and where $\epsilon_1=1$.
For a proof of the following lemma we refer to the Appendix of \cite{step2}. Throughout the rest of this section, we will denote by $B(x,r)$ the ball of centre $x$ and radius $r>0$ relative to the metric induced by the norm in \ref{eq:smoothbox}.

\begin{lemma}
  Provided the $\epsilon_\ell$ are small enough, the function \eqref{eq:smoothbox} defines a homogenous norm, and thus a metric, on $\mathbb{G}$. More precisely $\epsilon_\ell$ has only to be chosen small enough with respect to $\epsilon_1,\dots,\epsilon_{\ell-1}$.
\end{lemma}

The unit ball for this homogeneous norm is not isodiametric, indeed it is easy to find a strictly larger set diameter $2$.

\begin{proposizione}\label{prop:Gset}
Let $19/10<\xi<2$. Provided the $\epsilon_\ell$ are chosen small enough, the following set
  \[
    G(0,1) := \{ (x_1,\dots,x_\kappa),\  |x_1| < 1 \text{ and for } l=2,\dots, \kappa,\quad \epsilon_\ell^\ell |x_\ell| < \xi^{\ell-1}
    \},
  \]
  contains the unit ball and has diameter $2$.
\end{proposizione}

\begin{definizione}\label{def:ball}
For any $x\in\mathbb{G}$ and any $r>0$ we define $G(x,r):=x*\delta_r(G(0,1))$.
\end{definizione}

\begin{proof}[Proof of Proposition \ref{prop:Gset}]
Checking layer by layer, one sees that $B(0,1)\subset G(0,1)\subset B(0,2)$.
   It remains to prove that if $x,y$ are points in $G(0,1)$ there holds $||x^{-1}*y|| \le 2$. This is also done layer by layer: for the first layer, one immediately has $|y_1-x_1| \le 2$. Let $\ell\ge 2$.
   
   \medskip
   
  Recall that each entry of $Q_\ell:(V_1\oplus\ldots\oplus V_{\ell-1})\times(V_1\oplus\ldots\oplus V_{\ell-1})\to V_\ell$ is a polynomial in the components of $x$ and $y$, which is $\ell$ homogeneous with respect to dilations. Hence, following the footsteps of  \cite{step2}, we can find positive real numbers $c_{\ell,j}>0$ such that
\begin{equation}
    \begin{split}
    \label{Qestimate}
          |Q_\ell (x,y)| \le& \sum_{j=1}^{\ell-1} c_{j,\ell} ||x||^j||y||^{\ell-j}\leq \tilde{c}_\ell\sum_{j=1}^{\ell-1} \binom{\ell}{j} ||x||^j||y||^{\ell-j}\\
          \leq& \tilde{c}_\ell(||x||+||y||)^\ell\leq 2^{\ell-1}\tilde{c}_\ell(||x||^\ell+||y||^\ell),
    \end{split}
\end{equation}
  where the last inequality above comes from Jensen's inequality.
This implies, that if we let $c_\ell:=2^{\ell-1}\tilde{c}_\ell$, we can write
   \begin{equation*}
     \epsilon_\ell^\ell |(x^{-1}*y)_\ell|\le  \epsilon_\ell^\ell |y_\ell-x_\ell + Q_\ell(y,x^{-1})|\le \epsilon_\ell^\ell(|y_\ell|+|x_\ell| + c_\ell (||x||^\ell + ||y||^\ell)).
   \end{equation*}
   Using the fact that $x$ and $y$ are in $G(0,1)$ one gets 
   \[
     \epsilon_\ell^\ell |(x^{-1}*y)_\ell|\le  2\xi^{\ell-1} + \epsilon_\ell^\ell c_\ell (||x||^\ell + ||y||^\ell).
   \]
   As $G(0,1)\subset B(0,2)$, provided $\epsilon_\ell$ is small enough, the right hand side above is always bounded from above by $2^\ell$. Taking the $\ell$-th root concludes the proof of the claim.
 \end{proof}

 \begin{osservazione}\label{rem:eps-small}
\textbf{In the following, $\epsilon_\ell$ will be assumed to be small enough so that $4^\ell\epsilon_\ell^\ell c_\ell <(2-\xi)^{\kappa}$.}
 \end{osservazione}
 
\begin{proposizione}\label{prop:funz}
Let $0<\epsilon<1$. Assume $E$ is a set in $\mathbb{G}$ such that
\begin{equation}
    E\cap y*\{x\in\mathbb{G}:\lVert x\rVert>(1-\epsilon)^{-1} \lvert x_1\rvert\}=\emptyset \qquad\text{for any }y\in E.
    \label{parab}
\end{equation}
Then, there exists a continuous function $\varphi:\pi_1(E)\to V_2\oplus \ldots\oplus V_\kappa$ such that $\varphi(\pi_1(E))=E$ and in particular for any $w,z\in \pi_1(E)$ we have
\begin{equation}
\lVert\varphi(z)^{-1}\varphi(w)\rVert\leq (1-\epsilon)^{-1} \lvert z-w\rvert.
\label{eq:num:num:3}
\end{equation}
\end{proposizione}

\begin{proof}
First of all, we note that the projection $\pi_1$ is injective on $E$. To see this, let $y\in E$ and note that thanks to \eqref{parab} we have
\begin{equation}
    \begin{split}
        &\{(y_1,w_2,\ldots,w_\kappa)\in\mathbb{G}:\text{there exists an }2\leq \ell\leq \kappa \text{ such that }w_\ell\neq y_\ell\}\\
        &\qquad\qquad\subseteq \{w\in\mathbb{G}:\lVert y^{-1}*w\rVert>(1-\epsilon)^{-1}\lvert w_1-y_1\rvert\}\\
        &\qquad\qquad\qquad=y*\{x\in\mathbb{G}:\lVert x\rVert>(1-\epsilon)^{-1} \lvert x_1\rvert\}.
    \end{split}
\end{equation}
Since $\pi_1^{-1}(y_1)=(y_1,0)+V_2\oplus\ldots\oplus V_\kappa$, this concludes the proof of the injectivity of $\pi_1$ on $E$. This shows that there exists a function $\varphi:\pi_1(E)\to \mathbb{G}$ such that $\varphi(\pi_1(E))=E$. The continuity of $\varphi$ and \eqref{eq:num:num:3} come directly from \eqref{parab}. 
\end{proof}

Before passing to the proof of our main theorem, a technical result is needed:

\begin{proposizione}\label{prop:ballradii}
 There exists a continuous positive function $s:[0,1)\to (0,+\infty)$ such that if $x\in \mathbb{G}$ satisfies $|x_1|<||x||$, then there is $\lvert x_1\rvert<r<\lVert x\rVert$ such that $B(x,||x||s(|x_1|/||x||)))\subset G(0,r)\backslash B(0,r)$. Furthermore this function is bounded from above by $1-|x_1|/||x||$.
\end{proposizione}

\begin{proof}
  By homogeneity of the statement, it suffices to handle the case $||x||=1$.  Fix $x\in \mathbb{G}$ with $||x||=1$ and assume $|x_1|<1$. There exists an index $2\leq \ell_0\leq \kappa$ such that $\epsilon_{\ell_0} |x_{\ell_0}|^{1/{\ell_0}}=1$. 
  
  Fix an $s\in (0,1-|x_1|)$ and let $r \in(|x_1|+s,1)$. We want now to find conditions on $s$ and $r$ such that $B(x,s)\subset G(0,r)\backslash B(0,r)$. For any $y\in B(x,s)$ one has $|y_1| \le |x_1|+s < r$ and for every $\ell \in \{2,\dots,\kappa\}$ we have
  \begin{equation}
  \begin{split}
  \label{eq.estimates}
       s^\ell\geq& \epsilon_\ell^\ell\lvert (x^{-1}*y)_\ell\rvert=\epsilon_\ell^\ell\lvert-x_\ell+y_\ell+Q_\ell(-x_1,\ldots,-x_{\ell-1},y_1,\ldots,y_{\ell-1})\rvert\\
      \geq& \epsilon_\ell^\ell\lvert y_\ell\rvert-\epsilon_\ell^\ell\lvert x_\ell\rvert-\epsilon_\ell^\ell\lvert Q_\ell(-x_1,\ldots,-x_{\ell-1},y_1,\ldots,y_{\ell-1})\rvert\\
      \geq& \epsilon_\ell^\ell\lvert y_\ell\rvert-\epsilon_\ell^\ell\lvert x_\ell\rvert-\epsilon_\ell^\ell2^{\ell-1}\tilde{c}_\ell(||x||^\ell+||y||^\ell),
  \end{split}
  \end{equation}
  where in the last inequality we used \eqref{Qestimate}. This implies in particular that
  \[
    \epsilon_\ell^\ell|y_l| \le \epsilon_\ell^\ell |x_l| + s^\ell +\epsilon_\ell^\ell2^{\ell-1}\tilde{c}_\ell(||x||^\ell+||y||^\ell) \leq \epsilon_\ell^\ell |x_l|+s^\ell+4^\ell\epsilon_\ell^\ell\tilde{c}_\ell\leq 1+s^\ell+4^\ell\epsilon_\ell^\ell\tilde{c}_\ell.
  \]
  There also holds, thanks to \eqref{eq.estimates}, that
  \begin{equation}
      \begin{split}
        \epsilon_{\ell_0}^{\ell_0}|y_{\ell_0}| \ge& \epsilon_{\ell_0}^{\ell_0}|x_{\ell_0}|-s^{\ell_0}- 2^{\ell_0-1} \tilde{c}_{\ell_0} \epsilon_{\ell_0}^{\ell_0} (||x||^{\ell_0}+\lVert y\rVert^{\ell_0})\\
        \ge& (1-2^{\ell_0} \tilde{c}_{\ell_0} \epsilon_{\ell_0}^{\ell_0}) - s^{\ell_0}(1+2^{\ell_0} \tilde{c}_{\ell_0} \epsilon_{\ell_0}^{\ell_0}).  
      \end{split}
  \end{equation}
     The above computations shows that a sufficient condition to get the inclusion $B(x,s) \subset G(0,r)\backslash B(0,r)$ is
     \begin{equation}
         \begin{cases}
         |x_1|+s < r <1,\\
          1+s^\ell+4^\ell\epsilon_\ell^\ell\tilde{c}_\ell \le \xi^{\ell-1}r^\ell \text{ for } \ell =2,\dots,\kappa,\\
         (1-2^{\ell_0} \tilde{c}_{\ell_0} \epsilon_{\ell_0}^{\ell_0}) - s^{\ell_0}(1+2^{\ell_0} \tilde{c}_{\ell_0} \epsilon_{\ell_0}^{\ell_0}) >r^{\ell_0}.
         \label{eq:condizioni}
       \end{cases}
     \end{equation}

     By Remark \ref{rem:eps-small}, for each $\ell$, there holds  $1+ 4^\ell\epsilon_\ell^\ell c_\ell <\xi^{\ell-1}$ so when $r$ is close enough to $1$ there is an $s$ such that the conditions are satisfied. By continuity of the functions involved, there is a continuous positive function which to $|x_1|\in (0,1)$ associates a value  $s= s(|x_1|)$ such that there exists $r$ satisfying the above set of conditions. Thanks to \eqref{eq:condizioni}, we immediately see that the function $s(\lvert x_1\rvert)$ is bounded from above by $1-|x_1|$.
\end{proof}

\begin{lemma}
\label{lemma.regs}
In the notations of Proposition \ref{prop:ballradii}, 
for any $\varepsilon>0$ there exists a $k=k(\varepsilon)\in\N$ such that if $s(\lvert y_1\rvert/\lVert y\rVert)\leq 2/(k-1)$ then $\lvert y_1\rvert/\lVert y\rVert\geq 1-\varepsilon$.
\end{lemma}

 \begin{proof}
   Suppose this is not the case. Then, there exists an $0<\varepsilon_0<1$ and a sequence of $y_i\in \mathbb{G}$ such that
  $$s(\lvert (y_i)_1\rvert/\lVert y_i\rVert)\leq 1/i\qquad\text{and}\qquad \lvert (y_i)_1\rvert/\lVert y_i\rVert\leq 1-\varepsilon_0.$$
  Without loss of generality we can assume that $y_i\to y$ and by continuity of $s$ and of the norm, we would infer that
  $$s(\lvert y_1\rvert/\lVert y\rVert)=0\qquad\text{and}\qquad \lvert y_1\rvert/\lVert y\rVert\leq 1-\varepsilon_0,$$
  which however is in contradiction with the fact that $s$ is strictly positive on $[0,1)$. 
 \end{proof}

\begin{proposizione}\label{prop:graphproperty}
  Fix an $\epsilon>$ and choose $k=k(\varepsilon )\geq 2$ as in Lemma \ref{lemma.regs} and let $E$ be a Borel set of finite $\mathscr{H}^\alpha$-measure and such that
  \begin{equation}
      \lim_{r\to 0}\frac{\mathscr{H}^\alpha\llcorner E(B(x,r))}{r^\alpha}=1,\qquad \text{for }\mathscr{H}^\alpha\text{-almost every }x\in E.
      \nonumber
  \end{equation}
  For any $j\in\N$ let $E_{j,k}$ be the set of those $x\in E$ such that 
  \begin{equation}
      1-\frac{1}{k}\leq \dfrac{\mathscr{H}^\alpha (B(x,r)\cap E)}{(2r)^\alpha} \qquad {and}\qquad \dfrac{\mathscr{H}^\alpha(G(x,r)\cap E)}{(2r)^\alpha} \le 1+\frac{1}{k}, \qquad \text{for any }0<r<1/j.
  \end{equation}
  Then $\mathscr{H}^\alpha(E\setminus \cup_{j\in\N}E_{j,k})=0$, $\alpha\in\{0,1,\ldots,n_1-1\}$ and there exists a constant $L(k)>0$ such that the set $E_{j,k}$, can be covered $\mathscr{H}^\alpha$-almost all by images of $(1-\varepsilon)^{-1}$-bi-Lipschitz maps from a subset of $V_1$ endowed with the Euclidean metric to $\mathbb{G}$.
\end{proposizione}

\begin{proof}
Fix an index $j\in\N$ and cover $E_{j,k}$ let $B(y_i,r_i)$ be a countably many balls such that $r_i<(10j)^{-1}$ and $E_{j,k}\subseteq \cup_{i\in\N} B(y_i,r_i)$. Fix an $i\in\N$ let $x,y\in E_{j,k}\cap B(y_i,r_i)$ and without loss of generality, suppose that $x=0$. 

Suppose that $|y_1|<||y||$. Thanks to Proposition \ref{prop:ballradii} we know that  $B(y,s(|y_1|/||y||)||y||)\subset G(0,r)\backslash B(0,r)$ for some $r\le ||y||$. However, by definition of $E_{j,k}$, there holds
  \begin{equation}
        \mathscr{H}^\alpha(G(0,r)\backslash B(0,r) \cap E) \leq 2k^{-1}(2r)^\alpha,
        \label{estimate1}
  \end{equation}
  and
  \begin{equation}
      \mathscr{H}^\alpha \Big(B\Big(y,s\Big(\frac{|y_1|}{||y||}\Big)||y||\Big)\cap E\Big)\ge \Big(1-\frac{1}{k}\Big)\Big(2s\Big(\frac{|y_1|}{||y||}\Big)||y||\Big)^\alpha.
      \label{estimate2}
  \end{equation}
    In particular, putting together \eqref{estimate1} and \eqref{estimate2} we infer that
  \begin{equation}\label{eq:s-control}
  \begin{split}
       \Big(1-\frac{1}{k}\Big)s(|y_1|/||y||)^\alpha  \le& \frac{\mathscr{H}^\alpha\llcorner E(B(y,s(|y_1|/||y||)||y||))}{(2\lVert y\rVert)^\alpha}\\
       \leq&\frac{\mathscr{H}^\alpha\llcorner E(G(0,r)\backslash B(0,r))}{(2\lVert y\rVert)^\alpha}\leq \frac{2k^{-1}(2r)^\alpha}{(2\lVert y\rVert)^\alpha}\leq  2 k^{-1}.
  \end{split}
  \end{equation}
Thanks to our choice of $k$, we conclude that for any $i\in\N$ and $x,y\in B(y_i,r_i)\cap E$ we have
    \[
        ||y^{-1}*z|| \le (1-\varepsilon)^{-1} |(y^{-1}*z)_1|.
    \]
Lemma \ref{prop:funz} concludes the proof.
\end{proof}

\begin{proof}[Proof of Theorem \ref{thm:main}]\label{proofmain}
 As $\pi_1$ is 1-Lipschitz, $\pi_1(E)$ has finite Hausdorff measure, in particular, it has upper Hausdorff density at most $1$ almost everywhere.  Fix $\eta>0$, by Proposition \ref{prop:graphproperty}, for every $\eta>0$, $E$ can be covered up to a null set by images of $(1+\eta)$-biLipschitz maps $f_n^{\eta}$ defined on subsets $A_n^{\eta}$ of $V_1$. As $E$ has unit Hausdorff density almost everywhere, so do its subsets $f_n^\eta(A_n^\eta)$ at almost every point (see for instance Proposition 2.2 in \cite{antonelli2020rectifiableA}). Pick such a point $x\in f_n^\eta(A_n^\eta)$ (let us now write $f:= f_n^\eta$ and $A:= A_n^\eta$ for simplicity), there holds
 \[
    \lim_{r\to 0} \dfrac{ \mathscr{H}^\alpha(f(A) \cap B(x,r))}
    {(2r)^\alpha} = 1.
 \]
 As $\pi_1$ is $1$-Lipschitz, there holds
 \[
   \pi_1(f(A)\cap B(x,r)) \subset A \cap B(\pi_1(x),r)
 \]
 and since on the other hand $f$ is $(1+\eta)$-Lipschitz,
 \[
  \mathscr{H}^\alpha(f(A) \cap B(x,r)) \le (1+\eta)^\alpha \mathscr{H}^\alpha(A\cap B(\pi_1(x),r))
 \]
 and in turn
 \begin{align*}
 \dfrac{ \mathscr{H}^\alpha(\pi_1(E) \cap B(x,r))} {(2r)^\alpha} 
 &\ge
 \dfrac{ \mathscr{H}^\alpha(\pi_1(A) \cap B(\pi_1(x),r))} {(2r)^\alpha}\\
 &\ge \dfrac{1}{(1+\eta)^\alpha}
 \dfrac{ \mathscr{H}^\alpha(f(A) \cap B(x,r))} {(2r)^\alpha}.
 \end{align*}
 
 Letting $r$ tend to $0$, we obtain a bound on the lower density of $\pi_1(E)$:
 \[
    \Theta_*^\alpha(\pi_1(E),x) \ge (1+\eta)^{-\alpha}.
 \]

Thus, for every $\eta>0$ there is a subset of $\pi_1(E)$ of full measure with lower Hausdorff density at least $(1+\eta)^{-\alpha}$. Taking a countable intersection for a sequence of $\eta$ going to zero, yields a subset of $\pi_1(E)$ of full measure with unit Hausdorff density. Note that the Hausdorff measure and the density on $V_1$ are computed with respect to a euclidean norm. Therefore the main result of \cite{Mattila1975hausdorff} applies and $\pi_1(E)$ is euclidean rectifiable. Using the biLipschitz maps of Proposition \ref{prop:graphproperty}, one sees that $E$ itself is $\alpha$-rectifiable (in the metric sense of Federer, see \cite[\S 3.2.14]{Federer1996GeometricTheory}). In particular its tangents must be subgroups of $\mathbb{G}$ contained in $V_1$ and $\alpha$ must be the dimension of one of those subgroups.
\end{proof}

\section{The case of metric Lie groups admitting dilations}\label{sec:groupwithdilations}
This section contains a summary of results on metric Lie groups admitting a family of dilations and the construction of a biLipschitz equivalent group on which the conclusion of Theorem~\ref{thm:main} holds.

\begin{definizione}
  If $\mathbb{G}$ is a Lie group with Lie algebra $\mathfrak{g}$. A linear map $A:\mathfrak{g}\to \mathfrak{g}$ is a \textbf{derivation} if it obeys the Leibniz law: $A([x,y])= [A(x),y] + [x, A(y)]$. A derivation determines a one parameter family of dilations $(0,+\infty )\to \mathrm{Aut}(\mathbb{G}$:
  \[
     \lambda \mapsto \delta_\lambda := \lambda^A = e^{(\log \lambda) A}.
   \]
   A left-invariant distance $d$ on $\mathbb{G}$ is $A$-homogeneous if for every $\lambda \in (0,\infty)$: $d(\delta_\lambda x,\delta_\lambda y )=\lambda d(x,y)$.
 \end{definizione}
Note that if $A$ is diagonalizable in $\mathbb{R}$ with integer eigenvalues then $\mathbb{G}$ admits a $A$-homogeneous distance and is in fact a  homogeneous group.

\begin{teorema}[Theorem~1.6 in \cite{LeDonneNiGo2021dilations}]\label{thm:uniquetangent}
  Let $(X,d,\mu)$ be a metric measure space, which admits a unique tangent at $\mu$-almost every point. Then at $\mu$-almost every point, this tangent is a Lie group endowed with a $A$-homogeneous distance for some derivation $A$ of its Lie algebra.
\end{teorema}

Combining Theorem~1.2, Corollary~2.6 and Lemma~6.1 from \cite{LeDonneNiGo2021dilations}, yields:
\begin{proposizione}\label{prop:realspectrum}
  A group as in the conclusion of Theorem~\ref{thm:uniquetangent} is biLipschitz equivalent to a Lie group with a $A$-homogeneous distance where $A$ has real spectrum $\{t_1, \dots, t_\kappa   \}\subset[\,1,+\infty)$. Letting $V_{t_j}$ be the eigenspaces of $A$: then $A|_{V_1}$ is diagonalizable.
\end{proposizione}

The authors of \cite{LeDonneNiGo2021dilations} construct a $A$-homogeneous distance on such groups. In order to prove an analogue of Theorem~\ref{thm:main}, it is more practical to define a smooth-box distance instead. The following statement contains the properties of this distance which suffice to prove an analogue of Theorem~\ref{thm:main}.
\begin{teorema}
  If $(\mathbb{G},A)$ is as in Proposition~\ref{prop:realspectrum} then
  \begin{enumerate}[label = (\roman*)]
    \item \label{item:distanceexistsA} for each $j$, there exists a (euclidean) norm $||\cdot||_{t_j}$ on $V_{t_j}$ and a $A$-homogeneous distance $||\cdot||$ on $\mathbb{G}$ defined by
  \begin{equation}\label{eq:defnormA}
      || (x_{t_1},\dots , x_{t_\kappa})|| := \max \{ \epsilon_j ||x_{t_j}||^{1/t_j}, j=1,\dots, \kappa\},
  \end{equation}
    where $\epsilon_1= 1$ if $V_1\neq \{0\}$ and the other $\epsilon_{t_j}$ can be chosen arbitrarily small.
    \item\label{item:projcontinuousA} The projection $\pi_1:x\mapsto x_1 \in V_1$ is $1$-Lipschitz continuous.
    \item \label{item:defGA} If $B(x,r)$ denotes the balls for $||\cdot||$, then one can define $G(x,r)\supset B(x,r)$ as in Definition~\ref{def:ball}, and $G(x,r)$ has diameter $2r$.
  \end{enumerate}
\end{teorema}
\begin{proof}
  The proof of \ref{item:distanceexistsA} is a modification of that of Theorem~1.2 in \cite{LeDonneNiGo2021dilations}. One first choses a basis of $\mathfrak{g}$ in which $A$ is in Jordan form. Considering the spaces $V_{t_j}$ one by one and rescaling their bases, yields a euclidean norm on each $V_{t_j}$ such that for $\lambda\in (0,1\,]$ the operator norm of $\lambda^{A}|_{V_{t_j=}}$ is bounded from above by $\lambda^{t_j-\theta_j}$ for an arbitrarily small $\theta_j$. Using the fact that $A|_{V_1}$ is diagonalizable, one proves that the operator norm of $\lambda^{A}|_{V_1}$ is $\lambda$.
  Writing $||\cdot||$ as in \eqref{eq:defnormA}, it suffices to prove that the ''unit ball'' $B:=\{x,||x||\le 1\}$ is $A$-convex, or that for $\lambda\in (0,1)$ and $x,y$ in $B$, there holds
  \[
     || (\lambda^A x) \cdot ((1-\lambda)^Ay) || \le 1.
  \]
  This is done iteratively over the subspaces $V_{t_j}$. For $V_1$ it is clear. For $V_{t_j}$ with $1<t_j<2$, one proceeds as in Lemma~6.5 of \cite{LeDonneNiGo2021dilations}. For $t_j=2$, the reasonning of Lemma~6.7 in \cite{LeDonneNiGo2021dilations} works. For $t_j>2$, one needs to use the Baker-Campbell-Hausdorff formula as in the proof that the smooth-box norm on a Carnot Group is a metric (see the Appendix of \cite{step2}) but this is slightly more subtle (see Proposition~6.8 in \cite{LeDonneNiGo2021dilations}).

  \medskip

  In order to prove \ref{item:projcontinuousA}, pick $x= (x_1,x_{t_2},\dots, x_{t_\kappa})\in \mathbb{G}$. It suffices to prove on the one hand that
  \[
       ||  ((1/2)^A x) \cdot ((1/2)^A y) ||\le 1
  \]
  which is clear  by the Baker-Campbell-Hausdorff formula and the properties of the norms $||\cdot ||_{t_j}$,
  and on the other hand that iterating the procedure yields a sequence which tends to $(x_1,0,\dots)$. The fact that the components of order $t_j>1$ tend to $0$ is a consequence of $2 (1/2)^{t_j} <1$. And the fact taht the first component remains constant is clear because $A|_{V_1}$ is diagonalizable over $\R$ with eigenvalues $1$ (in fact, it is the identity). The proof of \ref{item:defGA} is straightforward using the properties of $||\cdot||$ and the technique of the previous section.
\end{proof}

\section{Discussion of the results in the first Heisenberg group}\label{sec:h1}

The point of this section is to compare 
the results of this work with the existing literature in the particular case of 
the first Heisenberg group. Throughout this section we will always endow $\HH^1$ with the Koranyi metric, which is the following homogeneous left invariant distance
\begin{equation}
    d_\mathscr{K}(x,y):=(\lvert x_1-y_1\rvert^4+\lvert x_2-y_2-2\langle x_1,Jy_1\rangle\rvert^2)^{1/4},
    \label{koranyidistance}
\end{equation}
where $J$ is the standard symplectic $2\times 2$ matrix. The fact that this defines a distance was first proved in \cite{Cygan1981Subadditivity}.

\medskip

Before delving into the discussion, we recall that whenever we endow $\mathbb{H}^1$ with a left invariant and homogeneous metric $d$, the metric space $(\mathbb{H}^1,d)$ is a $2,3,4$-purely unrectifiable metric space of Hausdorff dimension $4$, see \cite{Kirchheim2002UniformilySpaces}.
Another property of the Koranyi metric is that its ball, like that of the smooth-box metric, is not isodiametric. More precisely, in the following statement we construct a set which is a better candidate for the isodiametric problem.

\begin{proposizione}\label{GforKoranyi}
  Let $G(0,1)$ be the subset of those $x=(x_1,x_2)\in\HH^1$ such that $$\lvert x_1\rvert\leq 1\qquad \text{and}\qquad \lvert x_2\rvert^2 \leq 
  2*(1-\lvert x_1\rvert^4).$$
  Then $B(0,1)\subsetneq G(0,1)$ and $\diam(G(0,1))\leq 2$.
\end{proposizione}

The proof is a lengthy computation, which we omit. There is also an analogue of Proposition~\ref{prop:ballradii} for this choice of $B$ and $G$, which allows us to prove:
\begin{teorema}\label{thm:main2}
 Suppose $E$ is a Borel subset of $\mathbb{H}^1$ for which there exists an $\alpha>0$ such that $\mathscr{H}^\alpha(E) <+\infty$ and 
 \begin{equation}
     \lim_{r\to 0}\frac{\mathscr{H}^\alpha(B(x,r))}{r^\alpha}=1,\qquad\text{for $\mathscr{H}^\alpha$-almost all $x\in E$},
 \end{equation}
 and where $B(x,r)$ is the ball relative to the Koranyi distance, see \eqref{koranyidistance}. Then $\alpha=1$ and $E$ can be covered with countably many Lipschitz images of compact subsets of $\R$ up to an $\mathscr{H}^1$-null set. 
\end{teorema}

This is the analogue of  Mattila's beautiful result \cite{Mattila1975hausdorff} in which the author characterised rectifiable sets as those for which the density of the Hausdorff measure is equal to $1$ almost everywhere. Putting together \cite{kircharea} and Theorem \ref{thm:main2} we obtain, as in Theorem \ref{thm:main}, that rectifiable sets are characterised by having unit density for the Hausdorff measure almost everywhere. 
Furthermore, thanks to the work of G. Antonelli, V. Chousionis, V. Magnani, J. Tyson and the second named author the \emph{density problem}, see \cite{MarstrandMattila20} for a proper definition, has a complete and satisfactory solution in $\HH^1$ endowed with the Koranyi metric

\begin{teorema}[{\cite{antonelli2020rectifiableB,ChousionisONGROUP,Chousionis2015MarstrandsGroup,Merlo1,MarstrandMattila20}}]\label{preiss}
 Let $\phi$ be a Radon measure on $(\mathbb{H}^1,d)$ such that the limit 
 $$\lim_{r\to 0}r^{-\alpha}\phi(B(x,r))$$
 exists and is positive and finite $\phi$-almost everywhere, then $\phi\ll\mathscr{H}^\alpha$,  $\alpha\in\{0,1,\ldots,4\}$ and $\phi$ is $\mathscr{P}_\alpha$-rectifiable, or in other words
 \begin{itemize}
     \item[(i)] if $\alpha=0$ then $\phi$ is a sum of countably many multiples of Dirac masses;
     \item[(ii)] if $\alpha=1$, then $\mathbb{H}^1$ can be covered $\phi$-almost all by countably many Lipschitz images of $\R$;
     \item[(iii)] if $\alpha=2$, then $\phi$-almost everywhere the blowups of the measure $\phi$ coincide with the vertical line $\mathscr{V}:=\mathrm{span}\{e_3\}$, for a precise statement we refer to the introduction of \cite{antonelli2020rectifiableB};
     \item[(iv)] if $\alpha=3$, then $\mathbb{H}^1$ can be covered up to a $\phi$-null set 
     by countably many $C^1_{\mathbb{H}^1}$ 
     surfaces (see Definition~\ref{def:C1h}). Recall that $C^1_{\mathbb{H}^1}$-surfaces, introduced in \cite{Serapioni2001RectifiabilityGroup}, can be fractals from the Euclidean perspective and it is not even known at the moment of writing this work, if they are Lipschitz images of compact subsets of the parabolic plane, which is the model for homogeneous $1$-codimensional subgroups of $\mathbb{H}^1$.
     \end{itemize}
\end{teorema}

This result is the analogue of the celebrated Preiss's rectifiability theorem, see \cite{Preiss1987GeometryDensities}, in the first Heisenberg group $\mathbb{H}^1$. Theorems \ref{preiss} and \ref{thm:main2} show that there are both great similarities and differences between the general homogeneous groups and the Euclidean space. On the one hand, the Density Problem has an extremely natural solution in the form of Theorem \ref{preiss}. On the other, Theorem \ref{thm:main2} tells us that the very natural $C^1_{\mathbb{H}^1}$-regular surfaces \emph{cannot} have unitary density for the Hausdorff measure. This phenomenon however should be expected in view of Theorem \ref{thm:measuresrectifiable}, which tells us that in order to have unitary density for the Hausdorff measure at a point for a $C^1_{\mathbb{H}^1}$-surface we need that the intersection of the ball with the tangent plane at that point is an isodiametric set. Since tangents to $1$-codimensional surfaces are parabolic planes, the questions of Section \ref{sec:isodiametric} were motivated by trying to finding a metric for which there exists a $C^1_{\mathbb{H}^1}$-regular surface with unitary Hausdorff density.


\section*{Acknowledgements}
A.J.~wishes to thank Enrico Le Donne for interesting discussions. This work was completed at the HIM in Bonn during the 2022 Trimester Interactions between Geometric measure theory, Singular integrals, and PDE. We wish to thank the organizers of this event and the HIM staff. Both A.J.~and A.M.~were supported by the Simons Foundation grant 601941, GD. A.M was partially supported by the Swiss National Science Foundation
(grant 200021-204501 `\emph{Regularity of sub-Riemannian geodesics and
applications}')
and by the European Research Council (ERC Starting Grant 713998 GeoMeG
`\emph{Geometry of Metric Groups}').

%

\printbibliography
\end{document}